\documentclass{article}
\title{Affine Coxeter Extensions of the Two-Holed \\Projective Plane}

\usepackage{amsmath, amsthm,amssymb,bm,graphicx}
\usepackage{enumerate, mathrsfs}
\usepackage[all]{xy}
\renewcommand{\mathbf}{\bm}
\newcommand{\R}{\mathbb{R}}

\newcommand{\Isom}{\operatorname{Isom}}

\newcommand{\RP}{\mathbb{R}P}

\newcommand{\Z}{\mathbb{Z}}

\newcommand{\semidirect}{\ltimes}

\newcommand{\SO}{\operatorname{SO}}
\newcommand{\SL}{\operatorname{SL}}

\newcommand{\C}{\mathbb{C}}

\newcommand{\F}{\mathbb{F}}

\newcommand{\bmat}[1]{\ensuremath{\begin{bmatrix} #1   \end{bmatrix}}}

\newcommand{\GL}{\operatorname{GL}}

\newcommand{\Aff}{\operatorname{Aff}}

\newcommand{\Ad}{\operatorname{Ad}}

\renewcommand{\H}{\operatorname{\mathbb{H}}}

\newcommand{\Aut}{\operatorname{Aut}}

\newcommand{\tr}{\operatorname{tr}}

\renewcommand{\P}{\operatorname{\mathbb{P}}}

\newcommand{\PSL}{\operatorname{PSL}}

\newcommand{\Axis}{\operatorname{Axis}}

\newtheorem{thm}{Theorem}[section]

\newtheorem{exx}[thm]{Exercise}

\newtheorem{lem}[thm]{Lemma}

\newtheorem{prop}[thm]{Proposition}

\newcommand{\defeq}{\mathrel{\mathop:}=}

\author{William M. Goldman and Gregory D. Laun}
\newcommand{\spine}{\operatorname{Spine}}
\usepackage{xcolor}
\usepackage{todonotes}
\begin{document}
\maketitle

\begin{abstract}
A Margulis spacetime is a complete flat affine Lorentzian 3-manifold with free fundamental group.  Associated to $M$ is a noncompact complete hyperbolic surface $\Sigma$.  
We study proper affine actions of the double extension of $\pi_1 (M) \cong \pi_1 (\Sigma)$ when $\Sigma$ is homeomorphic to a projective plane minus two discs.   We classify such actions and show that there exist proper actions that do not admit crooked fundamental domains.
\end{abstract}

The goal of this paper is to classify affine orbifolds that are double-covered by a Margulis spacetime whose associated hyperbolic surface is homeomorphic to a two-holed cross surface (topologically, a plane minus two discs).

This is part of a larger project to study proper actions of  non-solvable discrete groups on affine spaces.  The case of a free group acting properly on $\R^3$ without fixed points is now fairly well understood.  The quotients of such an action are geodesically complete affine three-manifolds called \emph{Margulis spacetimes}.  Margulis spacetimes arise as infinitesimal deformations of hyperbolic surfaces with free fundamental group and such that the deformation uniformly lengthens or  shortens all closed geodesics~\cite{Goldman2009,goldman2000flat}.  In this way, each Margulis spacetime is canonically associated with a non-compact complete hyperbolic surface.  If $\Sigma$ is such a surface, and $M$ is the associated Margulis spacetime, we say that $M$ is an $\emph{affine deformation}$ of $\Sigma$.  Similarly, we say that the holonomy representation of $\pi_1 (M)$ is an affine deformation of the holonomy representation of $\pi_1 (\Sigma)$.
 
In this paper, we study proper actions by the involution group $\Psi \defeq \Z_2 \ast \Z_2 \ast \Z_2$.  As an abstract group, $\Psi$ naturally contains a free group $\F_2$ of rank two as an index two subgroup.  Moreover, every irreducible representation $\rho_0: \F_2 \to \Isom (\H^2)$ into the isometries of the hyperbolic plane admits a unique extension to a representation $\Psi \to \SL (2,\C)$, which we call a \emph{Coxeter extension}.

The case of orbifold quotients of affine actions is less well-studied than the manifold case.  Charette~\cite{charette2009groups} investigated affine deformations of reflection groups when the index two subgroup is the holonomy group of a three-holed sphere.  She showed the that there exist involution groups that act properly on affine three space which do not admit crooked fundamental domains, but whose index-two subgroups do admit crooked fundamental domains.  This is in contrast to the case of Margulis spacetimes: every proper action of a free discrete group of affine transformations on $\R^3$ admits a crooked fundamental domain~\cite{danciger2014margulis}.
The present paper is the first to investigate orbifold quotients when the corresponding hyperbolic manifold is non-orientable.

The basic strategy of this paper is similar to~\cite{CDG1HT}.   Let $\Sigma$ be a two-holed projective plane.  Given a hyperideal triangulation of $\Sigma$, we can use the theory of crooked ideal triangulations from~\cite{Burelle2012, CDG1HT} to realize this hyperideal triangulation as a configuration of crooked planes.

This process parametrizes the deformation space associated with a particular ideal triangulation of $\Sigma$.  In order to describe the full proper affine deformation space, we also need to describe how the deformation space changes when we change ideal triangulations.  Following~\cite{CDG1HT}, we use the \emph{flip graph} to encode an algebraic structure on the space of ideal triangulations on $\Sigma$.  By~\cite{danciger2014margulis} the image of the deformation space is actually the dual complex to the flip graph, the arc complex.  

In~\cite{charette2011finite}, it was shown that every proper affine deformation of a hyperbolic two-holed cross surface admits a crooked fundamental domain.  Moreover, the projectivized space of crooked fundamental domains was shown to be a quadrilateral $\mathbf{Q}$ in $\P ( H^1 (\Gamma_0,\R^{2,1})) \cong \RP^2$. The main goal of the present paper is to describe the projectivized space of crooked fundamental domains for a Coxeter extension of such a group.  Up to a choice of geodesic representatives for the linearized domain, this space is hexagon $\mathbf{H}$ inscribed in $\mathbf{Q}$.

Let $\Sigma$ be homeomorphic to a two-holed cross surface, and let $\pi \defeq \pi_1 (S)$ be the image of a choice of holonomy representation.  Let $\pi^{\prime}$ be a Coxeter extension of $\pi$. Choose a triangular fundamental domain for the action of $\pi^{\prime}$ on $\H^2$ such that the sides of the triangle are pairwise ultraparallel.  Two sides are determined by the two reflections that generate $\pi^{\prime}$: they are necessarily the fixed geodesics of these reflections.  The homotopy class of the remaining arc is fixed, but  there is an interval's worth of choice in the geodesic representative of this class.

Fix some choice of geodesic representative by picking a $\theta$ in the interval.  Call the corresponding fundamental domain $\mathscr{D}_{\theta}$.  Every other triangular fundamental domain for $\pi^{\prime}$ (with possibly a different set of generators) is given by $\tau \cdot \mathscr{D}_{\theta}$ for some $\tau$ a mapping class of $\Sigma$ and some choice of $\theta$.  Every crooked fundamental domain for a proper affine deformation of $\pi^{\prime}$ linearizes to a fundamental domain of the form $\tau \cdot \mathscr{D}_{\theta}$.  In what follows, we fix the parameter $\theta$.
\begin{thm}
\label{main}
The the space of proper affine deformations of $\pi^{\prime}$ that admit a crooked fundamental domain for fixed $\theta$ is a six-sided cone over the moduli space of hyperbolic structures on the orbifold quotient of $S$.  The cone projectivizes to a hexagon $\mathbf{H}$ inscribed in $\mathbf{Q}$.
\end{thm}
Varying $\theta$ gives an octagon instead of a hexagon, as two of the vertices are replaced by intervals.  See Figure~\ref{fig:octagon}.

As a corollary, there are proper affine deformations of $\pi$ that do not admit a crooked fundamental domain; namely those corresponding to points in $\mathbf{Q} \setminus \mathbf{H}$.  The corresponding fact was proved in the context of the three-holed sphere in~\cite{charette2009groups}.

\section{Notation}
\label{sec:notation}

Following Charette, Drumm, and Goldman~\cite{charette2011finite} and John H. Conway, we call the topological surface underlying a projective plane a ``cross surface''.   This is to avoid any confusion with the notion of the projective plane as a space carrying parabolic geometry.

We work in three-dimensional affine space.  Every Margulis spacetime is a quotient $\R^{2,1}/\Gamma$ where $\R^{2,1}$ is Minkowski space and $\Gamma$ is a free discrete subgroup of $\Aff (\R^{2,1}) = O (2,1)\semidirect \R^{2,1}$.  Let $G$ be a free group, and $\phi: G \to \Aff (\R^{2,1})$ the holonomy representation of a Margulis spacetime.  Projection onto the first factor gives a representation into $\SO (2,1) \cong \Isom^{\pm} (\H^2)$.  Call its image $\Gamma_0$.  Then $\Gamma_0$ can be identified with the holonomy group of a hyperbolic surface $\Sigma \defeq \H^2/\Gamma_0$.  In this paper we fix $\Sigma$ to be homeomorphic to a two-holed cross-surface.

Affine deformations $\Gamma$ of a fixed linear $\Gamma_0$ are classified by the cohomology $H^1 (\Gamma_0, \R^{2,1})$.  We  view $[u] \in H^1 (\Gamma_0,\R^{2,1})$ as assigning to each hyperbolic isometry $X \in \Gamma_0 \subset \Isom (\H^2)$ a translational part $u (X) \in \R^{2,1}$.  We call an element of $H^1 (\Gamma_0,\R^{2,1})$ an \emph{affine deformation}.  We call $[u]$ \emph{proper} if the semidirect product $\Gamma$ associated to $[u]$ acts properly on $\R^{2,1}$.  Additional details can be found in~\cite{Charette2010,Burelle2012,CDG1HT}.

We identify $\R^{2,1}$ with the Lie algebra $\mathfrak{psl} (2,\R)$.  The signature $(2,1)$ inner product, denoted by $\cdot$, is given by $1/2$ the trace form:
\[v \cdot u \defeq \frac{1}{2}\tr (vu)\]
We also need the Lorentzian cross product $\boxtimes$, defined  as the unique map satisfying
\[(u \boxtimes v) \cdot w = \det (u,v,w)\]
In particular, $u \boxtimes v$ is Lorentz-orthogonal to both $u$ and $v$.

As a Lie group, $SO (2,1)^0$ is isomorphic to $\PSL (2,\R)$. Since $\Sigma$ is non-orientable, we need orientation-reversing isometries as well.  These can be identified with matrices $iP$, $P \in \GL (2,\R)$, $\det (P) = -1$.  See~\cite{Goldm2009,charette2011finite}.

We identify $\H^2$ with the space of timelike subspaces in $\R^{2,1}$ in the standard way.  That is, each point in $\H^2$ defines a class $[t] \in \P (\R^{2,1} \setminus \left\{ 0 \right\})$ with $t \cdot t < 0$.
For a vector $u \in \R^{2,1}$, let $u^{\perp}$ denote the associated Lorentz-orthogonal subspace.  If $u$ is a spacelike vector, then $u^{\perp}$ is a linear plane that intersects the lightcone transversely, and so may be identified with a hyperbolic geodesic in $\H^2$.  The intersection of $u^{\perp}$ with the lightcone determines two future-pointing unit lightlike vectors $u^{\pm}$, which we think of as points on the ideal boundary $\partial \H^2$.  Specifically, we choose $u^{\pm}$ such that $\left\{ u^-, u^+,u \right\}$ is a right-handed basis of $\R^{2,1}$.  If $v$ and $w$ are spacelike vectors, we say that $v$ and $w$ are \emph{ultraparallel} if the corresponding hyperbolic geodesics defined by $v^{\perp}$ and $w^{\perp}$ are ultraparallel.
We say that spacelike vectors $v_1, v_2, v_3$ are \emph{consistently oriented} if
\begin{itemize}
\item $v_i \cdot v_j < 0$
\item $v_i \cdot v_j^{\pm} \leq 0$
\end{itemize}
whenever $i \neq j$.

Let $X \in \Isom (\H^2)$ be a hyperbolic or parabolic isometry.   The linear map defined by $A \mapsto XAX^{-1} = \Ad_X (A)$ has a 1-eigenspace that is spacelike if $X$ is hyperbolic and lightlike if $X$ is elliptic.  If $X$ is hyperbolic, choose a $1$-eigenvector $X^0$ of $X$ satisfying $ X^0\cdot X^0 = 1$.

Let $(X, u (X)) \in \Gamma$ with $X \in \Gamma_0$ hyperbolic.  The \emph{Margulis invariant} of the affine deformation $(X,u (X))$ is the neutral projection of the translational part $u$:
\[\alpha_{[u]} (X) \defeq u (X) \cdot X^0 \]
The map $g \mapsto \alpha_{[u]} (g)$ depends only on the cohomology class $[u]$ of $u$.  For any nonzero point $p \in \R^{2,1}$, the Margulis invariant of $X$ can be computed as $(XpX^{-1} - p) \cdot X^0$.

Every proper affine action by a non-solvable discrete group on $\R^{3}$ admits a fundamental domain bounded by \emph{crooked planes}~\cite{danciger2014margulis}. A crooked plane is a piecewise-linear surface invented by Drumm~\cite{Drumm1990} to enable ping-pong arguments in $\R^{2,1}$. A crooked plane is determined by a spacelike vector, called its \emph{direction vector}, and a point, called its \emph{vertex}.
Specifically, for a point $p$ and a spacelike vector $v$, define the crooked plane $\mathsf{C} (v,p)$ as follows.  It is the union of two \emph{wings}
\begin{align*}
p &+ \R_{+} v^+  + \R_+ v\\
p &+ \R_+ v^- - \R_+ v
\end{align*}
and a \emph{stem}
\[p + \left\{ x \in \R^{2,1} \mid v \cdot x =0, x \cdot x \leq 0 \right\}\]
By duality the direction vector corresponds to a hyperbolic geodesic $\ell$.  In the language of the Lie group $\PSL (2,\R)$ and its Lie algebra $\mathfrak{psl} (2,\R)$, a crooked plane is the set of all Killing fields with 
a non-repelling fixed point on $\ell$~(cf. \cite{danciger2014margulis}).

In order to build fundamental domains, we need to know when two crooked planes are disjoint.  The following criterion provides this information.
Let $w_1, w_2, w_3$ be unit spacelike vectors defining crooked planes. Let $u_i^{\pm} \in \R_{\geq 0}$ for $i \in \left\{ 1,2,3 \right\}$, and define the points
\begin{align*} 
q_1 &= u_1^{-} w_1^{-} - u_1^{+} w_1^{+}\\
q_2 &= u_2^- w_2^{-} - u_2^{+} w_2^{+}\\
q_{0} &= u_3^{-} w_3^{-} - u_3^+ w_3^{+}\\
\end{align*}
The following proposition follows from~\cite{Charette2010}. 
\begin{prop}
When all the coefficients $u_i^+$, $u_i^-$ are positive, then the crooked planes $\mathscr{C} (w_1,q_1)$, $\mathscr{C} (w_2,q_2)$, and $\mathscr{C} (w_{3},q_3)$ are disjoint.
\end{prop}
The parameters $u_i^{\pm}$ form a translation semigroup, called the \emph{stem quadrant}.  For a spacelike vector $w$, denote the stem quadrant by $V (w)$.  
\[\label{cocycle}V (w) = \R_+^{*} w^{-} - \R_+^{*} w^{+}\]
See~\cite{Burelle2012} for additional details.  The disjointness criterion can also be interpreted in the language of strip deformations, as in~\cite{danciger2014margulis}.

Finally, we need an analog of geodesic reflections in the language of affine deformations.  These are provided by \emph{spine reflections}.  Given a spacelike vector $u$, the corresponding spine reflection is a map $\spine (u) \in \PSL (2,\R)$ defined as: 
\[\spine (u):  v \mapsto -v + 2 \frac{v \cdot u}{u \cdot u}u \]
Charette studied spine reflections in~\cite{charette2009groups}.

\section{The space of hyperideal triangulations}
\label{sec:space-ideal-triang}

In this section, we build fundamental domains for the action of the involution group $\Psi$.

\subsection{A Fundamental Domain for the Action of $\Psi$}
\label{sec:2hx-fund-doma}

Let $\pi \cong \pi_1 (\Sigma)$ be the holonomy of a hyperbolic structure on $\Sigma$.  Then $\pi$ is a free group generated by two glide reflections $X,Y$ that intersect in a distinguished point $p_0$.  Additionally, $\Sigma$ has two boundary components $A,B \in \PSL (2,\R)$ which we can choose so that $A \defeq XY$, $B \defeq Y^{-1}X$.  This gives a redundant presentation 
\[\pi = \langle X, Y, A,B \mid A = XY, B = Y^{-1}X\rangle\]
\begin{figure}
  \centering
  \def\svgwidth{0.55\columnwidth} 
  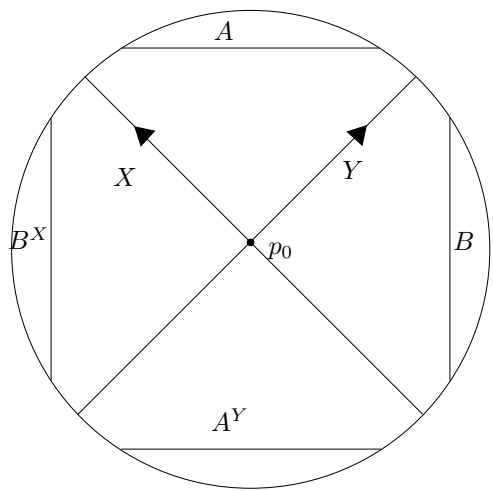
  \caption{A fundamental domain for the convex core of $\Sigma$.  For group elements $g,h$, the notation $g^h$ means $hgh^{-1}$.}
\label{fig:core}
\end{figure}

Let $\iota_0$ be the (orientation-preserving) point symmetry in $p_{0}$.  Then $\iota_0$ reverses the orientation of every geodesic passing through $p_{0}$.  Since $p_{0} = \Axis (X) \cap \Axis (Y)$, $\iota_0 X \iota_0$ is a glide reflection with the same translation distance and axis as $X$, but in the opposite direction. It follows that $\iota_0 X \iota_0  = X^{-1}$.  Similarly, $\iota_0 Y \iota_0 = Y^{-1}$.  Following~\cite{Goldman2009}, we see that $\langle X, Y, \iota_0 \rangle  \cong \Z_2 \ast \Z_2 \ast \Z_2$ is a Coxeter extension of $\pi$.

We can see better the structure of the involution group by defining $R_X \defeq X \iota_0$, $R_Y \defeq \iota_0 Y$.  The index two subgroup $\pi$ can then be recovered as:
\begin{eqnarray*}
\label{eq:index2}
X &=& R_X \iota_0\\
Y &=& \iota_0\ R_Y\\
A &=& R_X R_Y\\
B &=& R_Y \iota_0 R_X \iota_0
\end{eqnarray*}
While $\iota_0$ is a symmetry in a point, $R_X$ and $R_Y$ are reflections in hyperbolic geodesics.  Call these geodesics $\ell_X$ and $\ell_Y$ respectively.  
Then $\ell_Y$ is the mutual perpendicular of $\Axis (A)$ and $\Axis (B)$, and $\ell_X$ is the mutual perpendicular of $\Axis (A)$ and $\Axis (XBX^{-1})$.
 See Figure~\ref{fig:core}.

We depict two types of fundamental domain for the action of this Coxeter group in Figure~\ref{fig:cox-options}. Each is a hyperideal triangle bounded by geodesics that project in the quotient to $\ell_X$, and  $\ell_Y$ and by a third hyperbolic geodesic $\ell_0$ through $p_{0}$. The two triangulations differ by a diagonal flip that sends $\ell_0$ to a the line $\ell_0^{\prime}$ orthogonal to $\ell_0$ through $p_{0}$.
\begin{figure}
  \centering
  \def\svgwidth{0.55\columnwidth} 
  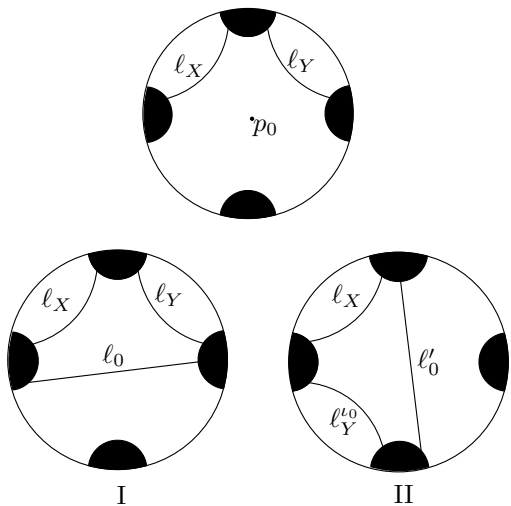
  \caption{The two fundamental domains for the Coxeter group.  Note that $\ell_Y$ and $\ell_Y^{\iota_0} = \iota_0 \ell_Y \iota_0$ both project to the same arc in the quotient.}
  \label{fig:cox-options}
\end{figure}

\subsection{Arc Complex of $\Sigma$}
\label{sec:arc-complex-sigma}

Recall the definition of the \emph{flip graph} of a surface $\Sigma$ with boundary.  The vertices of the flip graph are ideal triangulations of $\Sigma$, and there is an edge between two triangulations if and only if the two differ by a diagonal flip.  For surfaces with fundamental group free of rank 2, the dual complex to the flip graph is the \emph{arc complex}.  The vertices of the arc complex are homotopy classes of properly embedded arcs in $\Sigma$, and $k$ arcs span a simplex if and only if they can be realized disjointly.  The duality expresses the fact that a maximal collection of disjoint arcs defines an hyperideal triangulation of the surface.

Charette, Drumm, and Goldman~\cite{CDG1HT} used the flip graph (although using different language) to parametrize the proper affine deformation space of a one-holed torus.  Danciger, Gu{\'e}ritaud, and Kassel~\cite{danciger2014margulis} generalized this approach to use the arc complex to parametrize the proper affine deformation space of all convex-cocompact surfaces.

The arc complex for the two-holed cross-surface is depicted in Figure 3.  In the diagram, the antipodal points on the circle's boundary are identified.  The black semidisks (antipodally identified) indicate the two removed discs.  For some of the arc classes, we show two representatives for clarity.

An hyperideal triangle corresponds to a top-dimensional simplex of the arc complex (or alternately, to a point in the flip graph).  
In the present paper, we consider the hyperideal triangle of type $I$.  This forms a fundamental domain for the action of $\Psi$. There is an element of the mapping class group that interchanges $I$ with $II$.  Since this induces an automorphism of the fundamental group, $II$ is also a fundamental domain for the Coxeter extension.  The remaining simplices do not arise as fundamental domains for the Coxeter extension.

Algebraically, the flip between the two triangulations is achieved by an automorphism of $\pi$ that fixes $X$ and sends $Y$ to its inverse.  See the proof of~Proposition~\ref{pentagon-flip}.
\begin{figure}
  \centering
  \def\svgwidth{0.55\columnwidth} 
  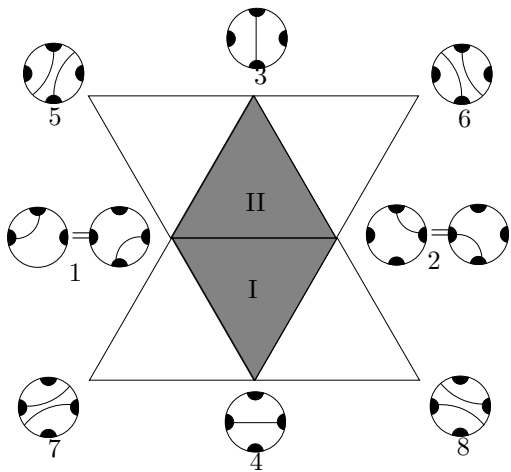 
    \caption{The arc complex of a two-holed cross surface.}
  \label{fig:arc-complex}
\end{figure}

%
\include{diagram.svg}

\subsection{Parametrization of Hyperideal Triangles}
\label{sec:param-ultr-triangl}

In this section, we work in a fundamental domain in configuration $I$.  We first parametrize hyperideal triangles in this configuration.

Without loss of generality, we assume that the axes of $X$ and $Y$ intersect at the origin $p_0 = 0$.  Let $\ell_0$ be a hyperbolic geodesic through $p_0$ making angle $\theta$ with the horizontal axis.
Let $m$ be the common perpendicular between $\ell_X$ and $\ell_Y$.  Let $d$ be the distance between $m$ and $p_0$, and let $u_1$ (respectively $u_2$) be the distance between $m$ and $\ell_X$ (respectively $\ell_{Y})$.

Given $u_1, u_2$, $d$, and $\theta$, define the spacelike vectors 
\[w_{i} =\begin{pmatrix} \cosh u_i\\ \sinh u_{i} \sinh d \\ \sinh u \cosh u_{i}\end{pmatrix}\] 
for $i = 1,2$, and 
\[w_{0} = \begin{pmatrix}\cos \theta\\ \sin \theta\\ 0\end{pmatrix}.\]  With appropriate bounds on the parameters, the geodesics $\left\{ w_i^{\perp} \right\}$ are disjoint and define a hyperideal triangle in configuration $I$.  The group $\Gamma_0 = \rho (\F_2)$ does not depend on $\theta$, only the choice of fundamental domain does.

\begin{figure}
  \centering
  \def\svgwidth{0.55\columnwidth}  
  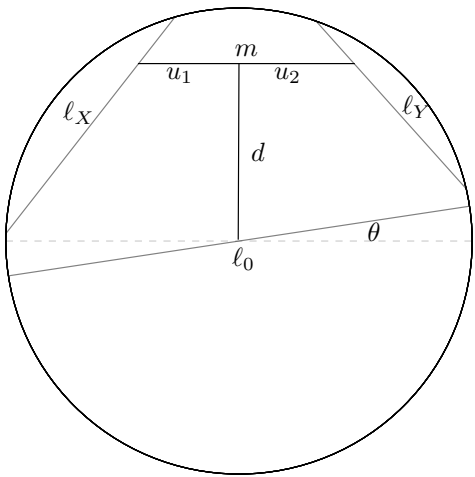 
    \caption{The parameter space for hyperideal triangles in configuration $I$.}
  \label{fig:hyperideal-param}
\end{figure}
Define the spine reflections
\begin{align*}
  \label{eq:spine-refl}
  R_X &\defeq \spine(w_{1})\\
  R_Y &\defeq \spine (w_2)
\end{align*}
and define $\iota_0$ to be the point symmetry in $p_{0}$.
Then $R_X$ and $R_Y$ are glide reflections whose axes intersect at $p_{0}$.

\section{Statement of Theorems}
\label{sec:statement-theorems}

From now on, fix the parameters $d, u_1, u_2, \theta$.  This also fixes $\Gamma_0$.  In order to determine which cocycles $[u] \in H^1 (\Gamma_0, \R^{2,1})$ act properly, we need the following facts about the Margulis invariant.

For fixed $g \in \Gamma$, the Margulis invariant determines a well-defined functional on $H^1 (\Gamma_0,\R^{2,1})$ 
\[\mu_g: H^1 (\Gamma_0,\R^{2,1}) \to \R \]
defined by 
\[\mu_g ([u]) \defeq \mu_{[u]} (g)\]
The invariants $X, Y$, and $A$ determine an isomorphism of vector spaces:
\[H^1 (\Gamma_0,\R^{2,1}) \cong \R^3,\]
given by 
\[[u] \mapsto \bmat{\mu_{[u]} (X)\\ \mu_{[u]} (Y) \\ \mu_{[u]} (A)}\]
 By abuse of notation, for each $g \in \Gamma_0$ denote its image $\rho (g) \in 
\Aff (2,1)$ by $g$ as well.
The following proposition was proved in~\cite{charette2011finite}.
\begin{prop}
\label{CDG2HX}
  Let $\Sigma$ be a two-holed cross surface with holonomy group $\pi = \rho \left (\pi_{1} (\Sigma) \right )$ with presentation as above.  Then an affine deformation $[u] \in H^1 (\Gamma_{0}, \R^{2,1})$ of $\pi$ acts properly on $\R^{2,1}$ if and only if the Margulis invariants $\alpha_{[u]} (X)$,$\alpha_{[u]} (Y)$,$\alpha_{[u]} (A)$, and  $\alpha_{[u]} (B)$ all are nonzero with the same sign.
\end{prop}

A triple of points $q_1, q_2, q_3$ in the stem quadrants $V (w_i)$ as in~(\ref{cocycle}) gives a cocycle $[u] \in H^1 (\Gamma_0, \R^{2,1})$.
Let $v \in \R^6$ be given by $v =
\begin{pmatrix}
  u_1^-\\ u_1^{+} \\   u_2^-\\ u_2^{+} \\   u_3^-\\ u_3^{+} \\
\end{pmatrix}$. 
Define the map $M: \R^6 \to \R^3$ by 
\[Mv \defeq
\begin{pmatrix}
\label{alpha-matrix}
  \alpha_{[u]} (X)\\ \alpha_{[u]} (Y)\\ \alpha_{[u]} (A)
\end{pmatrix}\]
Then $M$ is a linear map $M: \R^6 \to H^1 (\Gamma_0,\R^{2,1}) \cong \R^3$.

By Proposition~\ref{CDG2HX}, the image under $M$ of the positive orthant in $\R^6$ identifies with the space of crooked fundamental domains for the action of the affine deformation of $\pi$ corresponding to $[u]$ for a fixed choice of $\theta$. This image is a cone in $\R^3$.  Projectivized, this cone becomes a polygon in $\RP^2$.

In the present paper, we are interested not in the proper affine deformation space for $\Sigma$, but for its quotient $\Sigma^{\prime}$.  Define $\pi^{\prime}$ to be the Coxeter extension of $\pi$.  Since $\pi^{\prime}$ is a finite extension of a discrete group, it acts properly on $\R^{2,1}$ if and only if $\pi$ does. By~\cite{CDG1HT}, $\pi$ admits a crooked fundamental domain.  However, even if the action of $\pi$ admits a crooked fundamental domain, the action of $\pi^{\prime}$ may not.  In particular, we show

\begin{thm}[Hexagon]
\label{hexagon}
Let $\Gamma_0^{\prime}$ be a Coxeter group whose index two subgroup $\Gamma_0$ is the holonomy group for a two-holed cross surface.  Fix the parameter $\theta$.  This fixes a fundamental domain $\kappa$ for the linear Coxeter group.

Then the projectivized space of crooked fundamental domains for $\Gamma^{\prime}$ that linearize to $\kappa$ is a hexagon $\mathbf{H}$ in $\P (H^1 (\Gamma_0,\R^{2,1})) \subset \RP^2$.  The hexagon $\mathbf{H}$ is inscribed in the quadrilateral $\mathbf{Q}$ that parametrizes the space of crooked fundamental domains for the action of $\Gamma_0$.
\end{thm}
The hexagon is depicted in Figure~\ref{fig:hexagon}.
Allowing $\theta$ to vary gives an octagon, with the additional sides corresponding to the extra degree of freedom in choosing a geodesic representative for the arc preserved by the point symmetry.  See Figure~\ref{fig:octagon}.

The hexagon is the union of two pentagons.  Each pentagon is the projectivized image of the space of crooked fundamental domains in $\RP^2$ corresponding to a fixed ideal triangulation. The map that flips the two ideal triangulations induces a map on the space of crooked fundamental domains that interchanges the pentagons.  We thus prove Theorem~\ref{hexagon} in two steps.

\begin{thm}[Pentagon]
\label{pentagon}
Define the map $M$ as above.  The image of the positive orthant in $\R^6$ projectivizes to a pentagon $\mathbf{P_1}$ in $\RP^2$.
\end{thm}

\begin{figure}
  \centering
  \def\svgwidth{0.55\columnwidth}
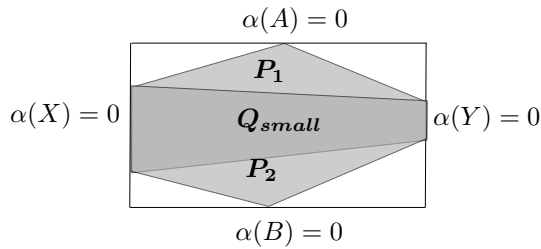  
\label{fig:hexagon}
  \caption{The hexagon $\mathbf{H}$ as the union of pentagons $\mathbf{P_1}$, $\mathbf{P_2}$ intersecting in $\mathbf{Q_{small}}$.}
\end{figure}

\begin{figure}
  \centering
\includegraphics[scale=.6]{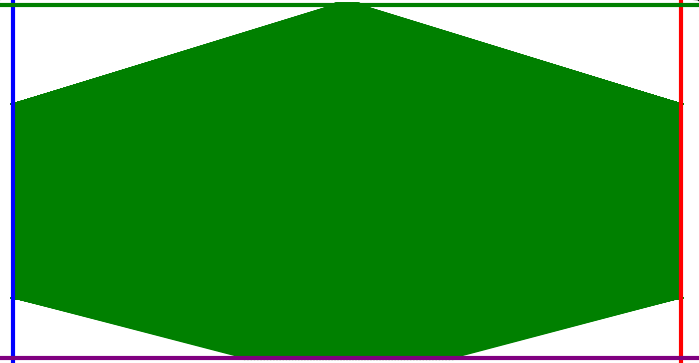}
  \caption{Allowing $\theta$ to vary gives a union of two hexagons, forming an octagon inscribed in $\mathbf{Q}$}.
\label{fig:octagon}
\end{figure}

Notice here the asymmetry between $\alpha (B)$ and $\alpha (A)$.  Specifically, $\alpha (A)$ can vanish, but there is no situation in which $\alpha (B) = 0$ while the other Margulis invariants remain nonzero.  This is an artifact of working with a fundamental domain in the configuration $I$.  Such a fundamental domain is asymmetric with respect to $A$ and $B$: it contains a self-loop at $B$ but not one at $A$.  We can recover the inherent symmetry of the problem by considering the map $\phi$ that achieves the diagonal flip $I \to II$.  The fundamental domain $II$ contains a self-loop at $A$ but not at $B$.

\begin{prop}
\label{pentagon-flip}
Consider an automorphism of the fundamental group $\phi \in \Aut (\pi)$ defined on generators by $\phi (X) = X$, $\phi(Y) = Y^{-1}$.  The automorphism extends to the Coxeter group as
\begin{align}
  \label{eq:1}
  \phi(R_X) &= R_X\\
\phi(R_Y) &= \iota_0 R_Y \iota_0\\
\phi(\iota_0) &= \iota_{0}
\end{align}
Then the image of $\phi$ is another inscribed pentagon $\mathbf{P_2}$ such that $\mathbf{P_1} \cap \mathbf{P_2}$ is a quadrilateral $\mathbf{Q_{small}}$ inscribed in the larger quadrilateral $\mathbf{Q}$.
\end{prop}
This automorphism switches the boundary components: $\phi (A) = B$, $\phi (B) = A$.

Together with the parametrization of hyperideal triangles, Theorem~\ref{hexagon} establishes Theorem~\ref{main}.

\section{Proofs of the Theorems}
\label{sec:proofs-two-theorems}

Denote the action of $\PSL (2,\R)$ on $\mathfrak{psl} (2,\R)$ by $(g,p) \mapsto g.p = gpg^{-1}$.

We now prove Theorem~\ref{pentagon} by explicitly computing $M$.  Recall the vectors $q_1,q_{2}, q_0$ defined above.
\begin{lem}
\label{alpha-compute}
We can compute the Margulis invariants as
\begin{align*}
\alpha (X) &= 2 (q_1 - q_0) \cdot X^0\\
\alpha (Y) &= 2 (q_0 - q_2) \cdot Y^0\\
\alpha (A) &= 2(q_1 - q_2) \cdot A^0\\
\alpha  (B) &= 2 (q_2 - \iota_0^{-1}q_1 \iota_0) \cdot B^0\\
\end{align*}
\end{lem}

\begin{proof}
We prove the result for $\alpha (B)$.  The rest are similar, but simpler.  Recall 
\[B = Y^{-1}X = (\iota_0 R_y)^{-1}X = R_y\iota_0^{-1}R_x\iota_0.\]  Also note that $\alpha (B) = (B.x - x, B^0)$ where $x \in \R^3$ is any nonzero vector.  It is convenient to compute $B.x$ where 
\[x = \iota_0^{-1} .q_1 = \iota_0^{-1}q_1 \iota_0.\]  This is because since $R_x$ fixes $q_{1}$, $\iota_0^{-1}R_x \iota_0$ fixes $\iota_0^{-1} q_1 \iota_{0}$.

\begin{align*}
(Y^{-1}X) (\iota_0^{-1}.q_1) &= \left( R_y \iota_0^{-1}R_x \iota_0 \right) (\iota_0^{-1}.q_1)\\
&= (R_y) (\iota_0^{-1}.q_1)\\
&= (R_y)\left( q_2 +  (\iota_0^{-1}.q_1 - q2)\right)\\
&= q_2 - (\iota_0^{-1}.q_1 - q2) \mod w_2
\end{align*}
Compare with~\cite{CDG1HT}. Since $w_2 \perp B^{0}$, this term vanishes in the computation of the Margulis invariant.

Then
\begin{align*}
\alpha (B) &= \left (B.(\iota_0^{-1}.q_1) - \iota_0^{-1}.q_1 \right ) 
\cdot B^0\\
&= \left ( q_2 - \iota_0^{-1}.q_1 + q_{2}  - \iota_0^{-1}.q_{1}\right ) \cdot B^0\\
&= 2 (q_2 - \iota_0^{-1}.q_1) \cdot B^0
\end{align*}
as desired.
\end{proof}

 View the direct sum of stem quadrants $V (w_1) \oplus V (w_2) \oplus V (w_0)$, as the positive orthant in $\R^6$.  We can decompose $M$ as
\[ M = M_1
\begin{pmatrix}
  u_1^-\\ u_1^{+} \\
\end{pmatrix}
+ M_2 \begin{pmatrix}
  u_2^-\\ u_2^{+} \\
\end{pmatrix}
+ M_3 \begin{pmatrix}
  u_3^-\\ u_3^{+} \\
\end{pmatrix}
\]
We now compute the matrices $M_1$, $M_2$, and $M_3$.

If $u_2^{\pm} = 0$ and $u_3^{\pm} = 0$, then $Mv = M_1 \begin{pmatrix}
  u_1^-\\ u_1^{+} \\
\end{pmatrix}$.  Let $e_1$, $e_2$ be the standard basis vectors of $\R^2$.  Then
\begin{align*}
\textsc{M}_{1}e_1 =
\begin{pmatrix}
  \alpha (X)\\
\alpha (Y)\\
\alpha (A)
\end{pmatrix}
& = 
\begin{pmatrix}
 2 (q_1 - q_0) \cdot X^0\\
 2 (q_0 - q_2) \cdot Y^0\\
2(q_1 - q_2) \cdot A^0
  \end{pmatrix}\\
&=
\begin{pmatrix}
 2 q_1 \cdot X^0\\
0\\
2 q_1 \cdot A^0\\
\end{pmatrix}
\\
&=\begin{pmatrix}
  2 w_1^- \cdot X^0\\
0\\
2w_1^- \cdot A^0
\end{pmatrix}
\end{align*}

Similarly, 
$M_{1}e_2 =
\begin{pmatrix}
  -2w_1^+ \cdot X^{0}\\
0\\
-2w_1^+ \cdot A^0
\end{pmatrix}
$.

For a general vector $
$ $u_1^- e_1 + u_1^{+} e_2$.  We compute $M_1$ 
\begin{align*}
 M_1 
&=
2 \begin{pmatrix}
 w_1^- \cdot X^0 & -w_1^+ \cdot X^0\\
0 & 0 \\
w_1^- \cdot A^0 & - w_1^+ \cdot A^0 
\end{pmatrix}
\end{align*}
The remaining matrices  $M_2$ and $M_3$ are analogous.  Explicitly:
\[M_2 = 2
\begin{pmatrix}
  0 & 0 \\
- W_2^- \cdot Y^0 & w_2^+ \cdot Y^0\\
- w_2^- \cdot A^0 & w_2^+ \cdot A^0\\
\end{pmatrix} 
\]

\[M_3 = 2\begin{pmatrix}
-w_3^- \cdot X^0 & w_3^+ \cdot X^0\\
 w_3^- \cdot Y^0&   -w_3^+ \cdot Y^0\\
0  & 0  
\end{pmatrix}
 \]

Define
\[v_{1 } \defeq \begin{pmatrix}
  u_1^-\\ u_1^{+} \\
\end{pmatrix}; v_{2 } \defeq \begin{pmatrix}
  u_2^-\\ u_2^{+} \\
\end{pmatrix}; v_{3 } \defeq \begin{pmatrix}
  u_3^-\\ u_3^{+} \\
\end{pmatrix}; 
\]

\begin{prop}
As above, define $M$ by 
\[Mv = M_1v_1 + M_2v_2 + M_3 v_3\]
Then $M_1$ and $M_2$ have rank 2 and $M_3$ has rank 1.  \end{prop}

\begin{proof}
  We first prove that $M_3$ has rank 1.  Recall the definition of $M_3$ as 
\[M_3 = 2
\begin{pmatrix}
-w_3^- \cdot X^0 & w_3^+ \cdot X^0\\
 w_3^- \cdot Y^0&   -w_3^+ \cdot Y^0\\
0  & 0  
\end{pmatrix}.
\]  Let $M_3^{\prime}$ be the $2 \times 2$ submatrix consisting of the nonzero entries in $M_3$.  Then we can write it as the product
\[M_3^{\prime} = 2\begin{pmatrix}
  X^0\\ Y^0
\end{pmatrix} \cdot
\begin{pmatrix}
  w_3^- &- w_3^+
\end{pmatrix}.
\]
The determinant is 
\begin{align*}
4(w_3^- \cdot X^0) (w_3^+ \cdot Y^0) - (w_3^- \cdot Y^0) (w_3^+ \cdot X^0) &= -4(X^0 \boxtimes Y^0 ) \cdot (w_3^- \boxtimes w_3^+) \\
&= -4(X^0 \boxtimes Y^0) \cdot w_3
\end{align*}
By construction, $w_3$ is a spacelike vector whose orthogonal space $w_3^{\perp}$ contains the timelike vector $p_0 = 
\begin{pmatrix}
  0\\0\\1
\end{pmatrix}
$ that spans the intersection of $(X^0)^{\perp} \cap (Y^0)^{\perp}$.  In particular, $w_3 \cdot (X^0 \boxtimes Y^0) = 0 $, and $M_3$ has determinant $0$.  Since $M_3$ is not the $0$ matrix, it has rank $1$.

We will show that $M_1$ has rank $2$.  The $M_2$ case is entirely similar.  As before, let $M_1^{\prime}$ be the $2 \times 2$ submatrix of $M_1$ consisting of nonzero entries.  As in the computation of $M_3^{\prime}$,
\[\frac{1}{4} \det M_1^{\prime} = -w_1 \cdot (X^{0} \boxtimes A^0).\]
Note that $X^{0} \boxtimes A^0$ is a spacelike vector since $X^0$ and $A^0$ are ultraparallel.  The spacelike vector spans the one-dimensional space $(X^0)^{\perp} \cap (A^0)^{\perp}$.  But as we noted above, the hyperbolic geodesic defined by the subspace $w_1^{\perp}$ is mutually perpendicular to the geodesics defined by  $(X^0)^{\perp}$ and $(A^0)^{\perp}$.  Thus 
\[X^0 \boxtimes A^0 = \lambda w_1\]
 for some nonzero $\lambda$.  Hence 
\[\frac{1}{4} \det M_1 = \lambda \|w_1\| \neq 0\]
and $M_1^{\prime}$ has full rank, so $M_1$ has rank $2$.
\end{proof}

\begin{prop}
Because $M_1$ and $M_2$ are matrices of full rank, the image of $e_1, \ldots, e_4$ are distinct.  Because $M_{3}$  is a rank-1 matrix, $\mathsf{M}e_5 = \mathsf{M}e_6$.  As a result, the image of the positive orthant in $\R^6$ projectivizes to a pentagon $\mathbf{P}_1$ in $\RP^2$.

\end{prop}
\begin{proof}
It is clear that $Me_5 = Me_6$, but the images $Me_i$ for $i \neq 5,6$ are distinct.  This gives $5$ distinct vectors in $\R^3$, which projectivize to  a pentagon in $\RP^2$.
\end{proof}

$\mathbf{P}_1$ is inscribed in $\mathbf{Q}$:
\begin{prop}
Let $\left\{ e_i \right\}$ be the standard basis vectors for $\R^6$.  Then
\begin{itemize}
\item $M e_1$ and $M e_2$  lie in $\ker \alpha_{[u]} (Y)$
\item $M e_3$ and $M e_4$ lie in $\ker \alpha_{[u]} (X)$
\item $M e_5$ and $M e_6$ lie in $\ker \alpha_{[u]} (A)$
\end{itemize}
In particular, the pentagon  $\mathbf{P}_1$ is inscribed in the quadrilateral $\mathbf{Q}$ defined by the projectivized images of the kernels of $\alpha_u (X)$, $\alpha_u (Y)$, $\alpha_u (A)$, $\alpha_u (B)$.
\end{prop}
\begin{proof}
This follows easily from the formulas.
\end{proof}
In the above proposition $Me_5 = Me_6$, and this point corresponds to the top of the pentagon in the diagram. The image of the other basis vectors gives a smaller quadrilateral $\mathbf{Q_{small}}$ inscribed in $\mathbf{Q}$.
To form the hexagon, we need to recover the corresponding singular point at the bottom.  We do this by proving Proposition~\ref{pentagon-flip}.
\begin{proof}~\ref{pentagon-flip}
Using a fundamental domain in configuration $II$ corresponds to using $\phi (R_X)$, $\phi (R_Y)$ and $\phi (\iota_0)$ as the generators.  This, in turn, corresponds to considering the image of the positive orthant in $\R^6$ under the matrix
\[M^{\phi } v \defeq  \begin{pmatrix}
\label{alpha-phi-matrix} 
  \alpha_{[u]} \phi (X)\\ \alpha_{[u]} \phi ( Y)\\ \alpha_{[u]} \phi ( A)
\end{pmatrix} = 
\begin{pmatrix}
  \alpha_{[u]} (X)\\ \alpha_{[u]} (Y)\\ \alpha_{[u]} (B)
\end{pmatrix} \]
The equality on $Y$ is due to the fact that the Margulis invariant satisfies $\alpha (Y^{-1}) = \alpha (Y)$ for any hyperbolic element $Y$.

Using $M^{\phi}$ in place of $M$ swaps the roles of $A$ and $B$ in Lemma~\ref{alpha-compute}.
Following the argument for $\mathbf{P_1}$, we get maps $M^{\phi}_1$, $M^{\phi}_2$ and $M^{\phi}_3$.  For any vectors $v,w$ $M_1^{\phi}v + M_2^{\phi}w = M_1v + M_2 w$. This image is the quadrilateral $\mathbf{Q_{small}}$.

Like $M_3$, $M_3^{\phi}$ has a one-dimensional image.  However, its image corresponds to a point on the line $\alpha_{[u]} (B) = 0$, giving another pentagon $\mathbf{P_2}$ with a vertex on the line in $\RP^2$ corresponding to the image of $\ker \alpha_{[u]} (B)$ and such that $\mathbf{P_1} \cap \mathbf{P_2} = \mathbf{Q_{small}}$.

\end{proof}

\bibliography{master}

\end{document}